\documentclass[11pt,oneside]{amsart}

  \usepackage{amssymb}
	\usepackage{epsfig}

\usepackage{epsfig}
\usepackage{epstopdf}
\usepackage{graphicx}
\usepackage{color}
\usepackage{mathptmx}
\usepackage{amsrefs}
\usepackage[pagewise]{lineno}

     \theoremstyle{plain}
      \newtheorem{theorem}{Theorem}[section]
      \newtheorem*{theorem*}{Theorem}
      \newtheorem*{theorem-MainThm1}{Theorem \ref{Main Thm 1}}
      \newtheorem*{theorem-MainThm2}{Theorem \ref{Main Thm 2}}
      
      \newtheorem{proposition}[theorem]{Proposition}
      \newtheorem{corollary}[theorem]{Corollary}
      
      \newtheorem{lemma}[theorem]{Lemma}
 
 \theoremstyle{definition}
      \newtheorem{remark}[theorem]{Remark}
      \newtheorem{definition}[theorem]{Definition}
      \newtheorem{example}[theorem]{Example}
    
      \newcommand{\Z}{\mathbb Z}
	
	\newcommand{\Q}{\mathbb Q}
      \newcommand{\nil}{\varnothing}
	
	\newcommand{\defn}[1]{\textbf{{#1} }}
	\newcommand{\boundary}{\partial}
	
	\newcommand{\inter}[1]{\mathring{{#1}}}
	\newcommand{\wihat}[1]{\widehat{{#1}}}
	
	\newcommand{\mc}[1]{\mathcal{{#1}}}
	\newcommand{\ob}[1]{\overline{{#1}}}
	\newcommand{\st}{\operatorname{st}}

      \makeatletter
      \def\@setcopyright{}
      \def\serieslogo@{}
      \makeatother
   
\begin{document}


   \title[Exceptional surgeries]{Exceptional surgeries on knots with exceptional classes}
   \author{Scott A. Taylor}

  \begin{abstract}We survey aspects of classical combinatorial sutured manifold theory and show how they can be adapted to study exceptional Dehn fillings and 2-handle additions. As a consequence we show that if a hyperbolic knot $\beta$ in a compact, orientable, hyperbolic 3-manifold $M$ has the property that winding number and wrapping number are not equal with respect to a non-trivial class in $H_2(M,\boundary M)$, then, with only a few possible exceptions, every 3-manifold $M'$ obtained by Dehn surgery on $\beta$ with surgery distance $\Delta \geq 2$ will be hyperbolic.
  \end{abstract}

   \date{\today}
\thanks{This paper originated from a conversation with Luis Valdez-S\'anchez at the School on Knot Theory and 3--Manifolds in honor of Francisco Gonzalez-Acu\~na at CIMAT in Guanajuato, Mexico. I am grateful to him for his encouragement and helpful comments. I am also grateful to Alex Rasmussen for helpful feedback on drafts of this paper and to Martin Scharlemann for the encouragement to revisit these arguments.}

   \maketitle


\section{Introduction}

A Dehn filling of a torus boundary component of a hyperbolic 3-manifold $N$ is \defn{exceptional} if it results in a non-hyperbolic 3-manifold $M$. Thurston showed that exceptional fillings truly are exceptional: for a fixed compact hyperbolic $N$, there are only finitely many exceptional fillings (see \cite{Agol, Lackenby-word}. As a result of the Geometrization Theorem \cite{Thurston, P1, P2, P3}, a compact, orientable, non-hyperbolic 3-manifold is either reducible, boundary-reducible, annular, toroidal, or a small Seifert fibered space. We say that the corresponding filling is reducible, boundary-reducible, etc. If we have two exceptional fillings we refer to them as X/Y for X and Y in the set 
\[\not\mathbb{H} = \text{\{reducible, boundary-reducible, annular, toroidal, small SFS\}}.\]

In the past few decades, a great deal of work has been done (see \cite{Go-survey} for a survey) to understand the relationship between distinct exceptional fillings on a fixed $N$.  Much of the work has focused on finding sharp upper bounds on the ``slope distance'' $\Delta$ between the filling slopes for particular choices of $X,Y \in \not\mathbb{H}$,  and then enumerating the specific situations in which the maximum value of $\Delta$ can be realized. Although much has been accomplished, there are still important open problems (for example the Cabling Conjecture \cite{GAS}.) Beyond the Cabling Conjecture, the situation when $X$ or $Y$ is toroidal has generated a lot of interest. In studying such fillings (see, for example, \cite{VS}), the nature of the essential torus and its relation to the core of the filling solid-torus often play an important role.  In this note, we show how a slight adaptation of methods used by Gabai \cite{G3} to prove the Property R and Poenaru conjectures; Scharlemann \cite{S2} in the study of the reducible/reducible and reducible/boundary-reducible cases; and Wu \cite{Wu-sutured} in the study of the reducible/annular case can be adapted to give very low upper bounds on slope distance for certain types of toroidal/Y fillings, for Y $\in \not\mathbb{H} - \{\text{small SFS}\}$. This paper has been written so that it should suffice as a gentle introduction to the sutured manifold techniques of Scharlemann (and subsequently, Wu). We have also made an effort to collect some well-known combinatorial techniques (e.g. Scharlemann cycles) so that the non-sutured manifold part of this paper is relatively self-contained. 

For convenience in stating our theorems, we establish the following definitions and conventions. $N$ is a compact, orientable 3-manifold with a torus boundary component  containing essential simple closed curves $a$ and $b$ intersecting minimally $\Delta \geq 1$ times.  The 3-manifolds $M'$ and $M$ are obtained by Dehn filling $\boundary N$ with slopes $a$ and $b$ respectively.  Usually $\alpha \subset M'$ and $\beta \subset M$ will denote the cores of the filling tori, although occasionally we allow them to be other 1-complexes. 

The \defn{$\beta$-norm} of a connected oriented surface $S \subset M$ transverse to $\beta$ is
\[
x_\beta(S) = \max(0, -\chi(S) + |S \cap \beta|).
\]
The $\beta$-norm of a disconnected surface is the sum of the $\beta$-norms of its components. This definition is, in fact, applicable whenever $\beta$ is a given 1-complex (not necessarily a knot). Perhaps the most important instance is when $\beta$ is taken to be the empty set, in which case we have
\[
x_\nil(S) = \max(0,-\chi(S))
\]
(for $S$ connected) and refer to it as the \defn{Thurston norm} of the surface $S$. The (possibly disconnected) surface $S$ is \defn{$\nil$-taut} if it is incompressible and minimizes Thurston norm; that is, if $S'$ is another oriented surface such that $[S',\boundary S]$ and $[S,\boundary S]$ are equal in $H_2(M,\boundary S)$ then $x_\nil(S) \leq x_\nil(S')$. Similarly, for any 1-complex $\beta$, the surface $S$ is \defn{$\beta$-taut} if $S-\beta$ is incompressible in $M -\beta$, if $S$ minimizes $\beta$-norm in the class $[S,\boundary S] \in H_2(M,\boundary S)$ and if $\beta$ always intersects $S$ with the same sign. 

By analogy with knots in solid tori, we can define the \defn{wrapping number} of a knot $\beta$ with respect to a class $\sigma \in H_2(M,\boundary M)$ to be the minimal geometric intersection number of $\beta$ with a $\nil$-taut surface representing $\sigma$ and the \defn{winding number} of $\beta$ with respect to $\sigma$ to be the absolute value of the algebraic intersection number of $\beta$ with a surface representing $\sigma$. Winding number is independent of the representative of $\sigma$.  Thus, for a class $\sigma \in H_2(M,T)$ the following two statements are equivalent:
\begin{itemize}
\item There is no representative of $\sigma$ which is both $\nil$-taut and $\beta$-taut.
\item The wrapping number and the winding number of $\beta$ with respect to $\sigma$ differ.
\end{itemize}
If $\sigma$ satisfies these statements we say that $\sigma$ is an \defn{exceptional class} in $M$ (with respect to $\beta$). We say that $(M,\beta)$ \defn{has an exceptional class} if there is an exceptional class in $H_2(M,\boundary M)$ with respect to $\beta$.

As a consequence of more general results, we prove:

\begin{theorem-MainThm2}[The exceptional surgery theorem]
Assume the following:
\begin{itemize}
\item $N$ is irreducible and boundary-irreducible
\item $H_2(M,\boundary M) \neq 0$, $M$ is irreducible, and $(M,\beta)$ has an exceptional class
\end{itemize}
Then the following hold:
\begin{enumerate}
\item $M'$ is boundary-irreducible.
\item If $M'$ has an essential sphere, then $M'$ has a lens space proper summand.
\item If $M'$ is irreducible and has an essential torus, then either $\Delta = 1$ or there is an essential torus bounding (possibly with a component of $\boundary M'$) a submanifold of $M'$ of Heegaard genus 2.
\item Assume that $M'$ is irreducible and atoroidal and that $\boundary M'$ has at most two components of genus 2 or greater. Then if $M'$ has an essential annulus, either $\Delta = 1$ or $\boundary M'$ is a single torus and $M'$ has Heegaard genus 2.
\end{enumerate}
\end{theorem-MainThm2}

\begin{remark}
Most all cable knots always have a longitudinal surgery producing a reducible 3-manifold with a lens space proper summand. Gordon and Luecke \cite[Theorem 3]{GL0} show that a reducing surgery on a knot in $S^3$ always produces a 3-manifold with a lens space proper summand. Our result proves an analogous theorem for knots $\beta \subset M$ having distinct winding and wrapping numbers with respect to a non-zero class in $H_2(M,\boundary M)$. Gordon and Luecke also show that a knot with irreducible exterior in $S^3$ \cite{GL-reducible1} or in a reducible 3-manifold \cite{GL-reducible2} with a non-trivial reducing surgery will have the surgery distance $\Delta = 1$. Does a similar result hold in our situation?
\end{remark}

\begin{remark}
It is always possible to create knots having toroidal surgeries by taking curves lying on sufficiently knotted genus two surfaces. The essential torus in the surgered manifold will intersect the core of the surgery torus  one or two times and the surgery distance (from the meridian of the knot) will be 1. Thus, the most interesting constructions of toroidal surgeries either have $\Delta \geq 2$ or have tori intersecting the core of the surgery torus three or more times.  Gordon and Luecke prove \cites{GL, GL2} that if surgery on a knot $\beta$ in $S^3$ with surgery distance $\Delta \geq 2$ produces a 3-manifold $M'$ with an essential torus, then $\Delta = 2$ and there is an essential torus in $M'$ intersecting $\alpha$ in exactly two points. They go on \cite{GL3} to show that, in fact, $\beta$ must be a Eudave-Mu\~noz knot \cite{EM}. Is a similar classification is possible in our case?
\end{remark}

\begin{remark}
An important special case of our theorem is the case when the exceptional class for $(M,\beta)$ can be represented by a torus. Such a torus will be ``non-positive'' (in the sense of \cite{GW-memoirs}) and ``non-polarized'' (in the sense of \cite{VS}). If such is the case, then the exterior of $\beta$ will have two toroidal fillings of distinct slopes.

 Gordon \cite{Go-conj} conjectured that if $N$ is hyperbolic and if $a$ and $b$ are slopes giving exceptional Dehn fillings, then $\Delta \leq 8$ and, unless $N$ is the exterior of the Figure 8 knot, then $\Delta \leq 7$. He also conjectured that unless $N$ is one of four particular manifolds, then $\Delta \leq 5$.  Lackenby and Meyerhoff \cite{LM} have shown the first part of Gordon's conjecture and Agol \cite{Agol} has shown that $\Delta \leq 5$ except for finitely many possibilities for $N$.  For particular types of exceptional fillings, sharp upper bounds on $\Delta$ are known (see, for example, \cite{Go-survey}.) Often, the maximum value for $\Delta$ is achieved only in very specific and identifiable situations, but elucidating precisely what happens is an area of active research. For example, Gordon \cite{Go-toroidal} has shown that for toroidal/toroidal exceptional surgeries $\Delta \leq 8$ and if $\Delta \geq 6$, then the situation is completely understood. Gordon and Wu \cite{GW-memoirs} continue this work and give a complete explanation of toroidal/toroidal exceptional surgeries with  $4 \leq \Delta \leq 5$. 
\end{remark}

Using the techniques we adapt for this paper, Scharlemann \cite{S2} produces very strong results concerning X/Y surgeries for X,Y $\in$ \{reducible, boundary-reducible\}. To simplify our exposition, we do not spend much time considering those possibilities and instead refer the reader to Scharlemann's paper \cite{S2}. In fact, this note could serve as an introduction to Scharlemann's work, since at present we are not concerned with tackling the Cabling Conjecture. Along the way, though, we use Heegaard splitting theory to improve the conclusions of some of the results. The papers \cite{LOT, Kang, GW-memoirs}, among others, use non-sutured manifold techniques to consider X/Y for X,Y $\in$ \{reducible, boundary-reducible, annular, toroidal\} and obtain larger bounds for $\Delta$. In many cases, the bounds obtained in those papers are sharp, and the situations realizing the maximal values of $\Delta$ are completely classified. In \cite{Wu-sutured}, Wu uses sutured manifold theory to consider reducible/annular exceptional surgeries. Theorems 3.3 and 3.4 from that paper consider X/Y exceptional surgeries with X $\in$ \{reducible, boundary-reducible\} and Y $\in$ \{toroidal, annular\}. Under slightly different hypotheses from those of Theorem \ref{Main Thm 2} above, Wu also obtains $\Delta = 1$. The sutured manifold techniques used in the proofs of those theorems are very different from those in this paper and rely heavily on essential laminations. Using sutured manifold techniques similar to those in this paper (and, of course, in \cite{S1,S2}), and a detailed analysis of Gabai discs, Wu, under certain mild hypotheses, obtains the bound $\Delta \leq 2$ for reducible/annular exceptional surgeries. This present paper can also be seen as an introduction to that part of Wu's paper.

Our results are also closely connected with a famous result \cite[Corollary 2.4]{G2} of Gabai's: If $T$ is a boundary component of  a compact, orientable, irreducible, atoroidal 3-manifold $N$ having $\boundary N$ the union of tori and if $H_2(N, \boundary N - T) \neq 0$, then at most one Dehn filling of $T$ can reduce the Thurston norm of a class of $H_2(N,\boundary N - T)$ or make $N$ reducible. Lackenby \cite{L} used Gabai's work as the basis for showing that, under certain somewhat technical hypotheses, if $\alpha \subset M'$ is a null-homologous knot and if a reducible 3-manifold $M$ is obtained by surgery on $\alpha$, then any closed surface $\ob{Q} \subset M'$ can be isotoped so that
\begin{equation}\label{Lackenby Ineq}
(\Delta - 1)|\ob{Q} \cap \alpha| \leq -\chi(\ob{Q})
\end{equation}
where $\Delta$ is the intersection number between the surgery slope and the meridian of $\alpha$. We observe that the statement that the Thurston norm of a class of $H_2(N, \boundary N - T)$ drops after Dehn filling is equivalent to the statement that this class is an exceptional class for $(M,\beta)$ where $\beta$ is the core of the filling solid torus. Thus, our results are in some sense a generalization of Lackenby's. In particular,  the assumption that $\beta$ is null-homologous in $M$ can be dropped.  It comes at the cost, however, that we need to consider the surfaces $\ob{Q} \subset M'$ up to the relation of ``rational cobordism'' which is more general than isotopy.  Our version of Lackenby's theorem is most naturally stated in terms of sutured manifold theory, and we do this in the next section. Here, we content ourselves with defining rational cobordism stating a simple consequence.

Two properly embedded oriented surfaces $\ob{Q}$ and $\ob{R}$ are \defn{rationally cobordant} if $\boundary \ob{Q} = \boundary \ob{R}$, if their interiors are disjoint, and if they cobound a 3-dimensional submanifold $W \subset M'$ such that the inclusions of each of $\ob{Q}$ and $\ob{R}$ into $W$ induce isomorphisms of rational homology groups. We extend the definition of  ``rationally cobordant'' to make it an equivalence relation on properly embedded surfaces. A surface $\ob{Q} \subset M'$ is \defn{rationally inessential} if it is rationally cobordant to an inessential surface. 

\begin{theorem-MainThm1}
Assume (A) - (D) in Section \ref{Apps} and also that $M$ is reducible or that $(M,\beta)$ has an exceptional class. Suppose that $M'$ contains a rationally essential closed surface $\ob{Q}$ of genus $g$. Then:
\[
(\Delta - 1)|\ob{Q} \cap \alpha| \leq -\chi(\ob{Q})
\]
\end{theorem-MainThm1}

\begin{remark} Assumptions (A) - (D) are similar to the hypotheses of Theorem \ref{Main Thm 2}. They rule out a narrow class of possibilities for $M$, $M'$, and $N$ and require that $H_2(M,\boundary M) \neq 0$. 
\end{remark}

Most of what we do applies not only to questions about Dehn filling, but also to questions about 2-handle addition. The techniques in this paper as applied to 2-handle addition are elaborated and generalized in \cite{T1}. A more directly analogous version of Lackenby's result for 2-handle addition is obtained in \cite{T2} (using much more sophisticated sutured manifold theory) and several applications of that result are given in \cite{T3}.

For reference, we extend the definition of ``exceptional class'' to arcs.
\begin{remark}
If $\beta \subset M$ is a properly embedded arc, a class $\sigma \in H_2(M,\boundary M)$ is \defn{exceptional} if no representative of $\sigma$ is both $\beta$-taut and $\nil$-taut. We cannot rephrase this in terms of wrapping and winding number, since we may be able to change the algebraic intersection number between $\beta$ and a surface representing $\sigma$ by isotoping $\boundary S$ across an endpoint of $\beta$.
\end{remark}

\section{Sutured Manifold Theory}\label{SMT}
Although most of the arguments in this paper are classical, their applications to questions about exceptional surgeries do not seem to be well-known. We have endeavored to write this paper so that those without prior exposure to sutured manifold theory can follow the essentials of the arguments.  For more detail or for definitions not given here see \cites{G1, G2, G3, S1, S2, S3}.

\subsection{Sutured Manifolds and Generalized Thurston Norms}

\begin{definition}
A \defn{sutured manifold} $(M,\gamma,\beta)$ consists of a compact, orientable 3-manifold $M$, a properly embedded 1-complex $\beta$ and a (possibly empty) collection of oriented simple closed curves $\gamma \subset \boundary M$ with regular neighborhood $A(\gamma)$ such that:
\begin{itemize}
\item The complement of the interior of $A(\gamma)$ in $\boundary M$ is the union of three (possibly disconnected) disjoint surfaces denoted $R_-(\gamma)$, $R_+(\gamma)$, and $T(\gamma)$.
\item $T(\gamma)$ is the union of tori.
\item Each component of $\gamma$ is adjacent to a component of both $R_-(\gamma)$ and a component of $R_+(\gamma)$.
\item $R_-(\gamma)$ has inward pointing normal orientation.
\item $R_+(\gamma)$ has outward pointing normal orientation
\item The orientation of each component of $\gamma$ coincides with the orientation on it induced by $R_-(\gamma)$ and $R_+(\gamma)$
\item $\beta$ is disjoint from the curves $\gamma \cup \boundary A(\gamma)$.
\end{itemize}
\end{definition}

\begin{example}
A compact, orientable 3-manifold $N$ whose boundary is the union of tori can be thought of as a sutured manifold $(N,\gamma,\nil)$ where $\boundary N = T(\gamma)$ and $\gamma = \beta = \nil$. If we Dehn fill a component of $\boundary N$, obtaining a 3-manifold $M$, and let $\beta$ be the core of the filling torus, then $(M, \gamma, \beta)$ is a sutured manifold with $\gamma = \nil$.
\end{example}

\begin{definition}
A sutured manifold $(M,\gamma,\beta)$ is \defn{$\beta$-taut} if:
\begin{itemize}
\item $M - \beta$ is irreducible
\item No edge of $\beta$ has both endpoints in $R_-$ or both endpoints in $R_+$
\item $R_-(\gamma)$, $R_+(\gamma)$, and $T(\gamma)$ are all $\beta$-taut.
\item $\beta$ is disjoint from $A(\gamma) \cup T(\gamma)$
\end{itemize}
\end{definition}

\begin{example}
If $N$ is a compact, orientable, irreducible 3-manifold with $\boundary N$ the union of tori, then $(N,\nil,\nil)$ is a $\nil$-taut sutured manifold. If $M$ is the result of Dehn filling a boundary component of $N$ and if $\beta$ is the core of the filling solid torus, then $(M,\nil,\beta)$ is a $\beta$-taut sutured manifold. The manifold $(M,\nil,\nil)$ will be $\nil$-taut if and only if $M$ is irreducible.
\end{example}

\begin{example}
If $(N,\gamma,\beta)$ is a sutured manifold and if $M$ is the result of attaching a 2-handle to a curve $\gamma'$ in $\gamma$, then $(M,\gamma - \gamma', \beta \cup \beta')$ is another sutured manifold, where $\beta'$ is the cocore of the 2-handle. The sutured manifold $(N,\gamma,\beta)$ is $\beta$-taut if and only if the sutured manifold $(M,\gamma - \gamma', \beta \cup \beta')$ is $(\beta \cup \beta')$-taut \cite[Lemma 2.3]{S2}.
\end{example}

The following result of Scharlemann will be useful to us when we apply a sutured manifold result outside the category of suture manifolds.

\begin{proposition}[{\cite[Prop. 5.9]{S2}}]\label{Scharl-sutures}
Let $N$ be an irreducible, boundary-irreducible, compact, orientable 3-manifold and suppose that $\boundary N$ contains at most two non torus components. If:
\begin{enumerate}
\item If $c \subset \boundary N$ is a simple closed curve then there exists a taut sutured manifold structure on $N$ for which $c \subset R(\gamma)$.
\item If $c_1$ and $c_2$ are disjoint, essential, simple closed curves on $\boundary N$ which are non-coannular, then there is a taut sutured manifold structure for which $c_1$ and $c_2$ lie in $R(\gamma)$.
\end{enumerate}
\end{proposition}

Lackenby adapts Scharlemann's methods and proves:
\begin{proposition}[{\cite[Theorem 2.1]{L2}}]\label{Lack-sutures}
Let $N$ be a compact orientable irreducible 3-manifold, possibly with boundary. Then there is a taut sutured manifold structure on $N$.
\end{proposition}

In \cite[Lemma 4.1]{T1}, the present author combines Scharlemann and Lackenby's work and produces yet another method for placing sutures, one which is particularly well-adapted to problems concerning 2-handle addition. We will not need it in this paper.

\subsection{Parameterizing Surfaces}

Parameterizing surfaces are objects that allow sutured manifold theory to be applied to questions about the intersection between surfaces and knots or arcs. Inequalities such as Lackenby's inequality \eqref{Lackenby Ineq} or the inequality in Theorem \ref{Main Thm 2} arise from the ``index'' or ``sutured manifold norm'' of a parameterizing surface.

Suppose that $(M,\gamma,\beta)$ is a sutured manifold and let $N = M - \inter{\eta}(\beta)$ be the exterior of $\beta$. For each edge $e$ (but not loop) of $\beta$, let $A(e) \subset \boundary\eta(\beta)$ be the corresponding meridional annulus. For a simple closed curve $\alpha$ in $\boundary N$, we let $\mu(\alpha)$ denote the total number of spanning arcs in $A(e)$, over all edges of $\beta$.

A properly embedded orientable surface $Q \subset N$ is a \defn{parameterizing surface} if the following hold:
\begin{enumerate}
\item[(P1)] $\boundary Q$ is transverse to $\gamma$
\item[(P2)]  if $e$ is an edge of $\beta$, then each component of $\boundary Q \cap A(e)$ is either a circle or a spanning arc
\item[(P3)] no component of $Q$ is a sphere or disc disjoint from $\gamma \cup \boundary \eta(\beta)$. 
\end{enumerate}
The index $I(Q)$ of a parameterizing surface $Q$ is defined to be:
\[
I(Q) = -2\chi(Q) + |\boundary Q \cap \gamma| + \mu(Q)
\]
where $\mu(Q)$ is, by definition, $\mu(\boundary Q)$. 

\begin{example}
Suppose that $M$ is a compact, orientable 3-manifold with $\boundary M$ the union of tori and that $\beta \subset M$ is a link, so that $(M,\nil,\beta)$ is a sutured manifold. If $Q$ is any parameterizing surface in $(M,\nil,\beta)$, then $I(Q) = -2\chi(Q)$. Consequently, if $\ob{Q}$ is an orientable surface in a 3-manifold $M'$ obtained by Dehn surgery on $\beta$, then (as long as $\ob{Q}$ is not a sphere or disc contained in $N= M - \inter{\eta}(\beta)$) $Q = \ob{Q} \cap N$ is a parameterizing surface for $(M,\nil,\beta)$ with index equal to $I(Q) = -2\chi(\ob{Q}) + 2|\ob{Q} \cap \alpha|$, where $\alpha$ is the core of the surgery solid torus.
\end{example}

\begin{example}
Suppose that $(N,\gamma,\nil)$ is a sutured manifold and that $Q \subset N$ is a parameterizing surface such that all components of $\boundary Q$ are parallel to a simple closed curve $a \subset \boundary N$. Let $b$ be a component of $\gamma$, let $M$ be the result of attaching a 2-handle to $\boundary M$ along $b$ and let $\beta$ be the cocore of the 2-handle. Let $\Delta$ be the intersection number between $a$ and $b$. Then $Q$ is a parameterizing surface for $(M,\gamma-b,\beta)$. If $a$ is disjoint from all components of $\gamma - b$, then $I(Q) = -2\chi(Q) + |\boundary Q|\Delta$, whether $Q$ is thought of as a parameterizing surface in $(N,\gamma,\nil)$ or in $(M,\gamma-b,\beta)$. If $M'$ is the result of attaching a 2-handle to $\boundary N$ along $a$ and if $\ob{Q}$ is the result of capping off $\boundary Q$ with discs, then \[I(Q) = -2\chi(\ob{Q}) + 2|\boundary Q| + |\boundary Q|\Delta.\] Consequently, if, by some chance, $I(Q) \geq 2\mu$ where $\mu$ is the number of intersections between $\boundary Q$ and $b$, then we obtain an inequality similar to Lackenby's \eqref{Lackenby Ineq}.
\end{example}

\subsection{Gabai Discs}
The final definition we need is not, strictly speaking, in the domain of sutured manifold theory, but it is convenient to place it here.

\begin{definition} (cf. \cite[Section 3]{S2})
Suppose that $(M,\gamma,\beta)$ is a sutured manifold with $\beta$ an oriented 1--manifold. Let $Q$ be a parameterizing surface in $(M,\gamma,\beta)$ traversing each edge or loop of $\beta$ at least $\mu$ times. Suppose that $D \subset M$ is an embedded oriented disc transverse to $\beta$. $D$ is a  \defn{Gabai disc} for $Q$ in $M$ if 
\begin{itemize}
\item $|\beta \cap D| = q > 0$ and all points of intersection have the same sign.
\item $|Q \cap \boundary D| < \mu.$
\end{itemize}
\end{definition}

As we explain in Section \ref{Gabai Discs}, when $\beta$ is a knot or arc then the presence of a Gabai disc produces a so-called ``Scharlemann cycle''. The Scharlemann cycle can be used to simplify the parameterizing surface. The simplification procedure gives rise to the notion of ``rationally cobordant'' mentioned earlier.

\section{Sutured Manifold Hierarchies}

\subsection{Fundamentals}
The fundamental tool for studying sutured manifolds is the sutured manifold hierachy, which is essentially a technical adaptation to the category of sutured manifolds of the usual sort of hierachy in 3-manifold theory. In brief, given a sutured manifold $(M,\gamma,\beta)$ and a parameterizing surface $Q \subset M$, we cut $M$ along certain types of oriented surfaces $S$ in such a way that the manifold inherits a natural sutured manifold structure $(W,\gamma',\beta')$ and the parameterizing surface $Q$ gives rise to a parameterizing surface $Q' \subset W$. See Figure \ref{Fig:  SuturedDecomp} for a schematic depiction. There are several (slightly different) notions of suture manifold hierarchy, the version we use in this paper is that found in \cite{S1} as modified in \cite{S2}. The 3-manifold $W = M - \inter{\eta}(S)$. The sutures $\gamma'$ are essentially the double curve sum of the sutures $\gamma$ with the oriented curves $\boundary S$. The 1--complex $\beta' = \beta \cap M'$ and the parameterizing surface is, roughly speaking, $Q \cap W$, although certain types of boundary compressions and isotopies may have to be performed on $Q$ prior to the decomposition to guarantee that $Q'$ remains a parameterizing surface. 

\begin{figure}[tbh]
\centering
\includegraphics[scale=0.4]{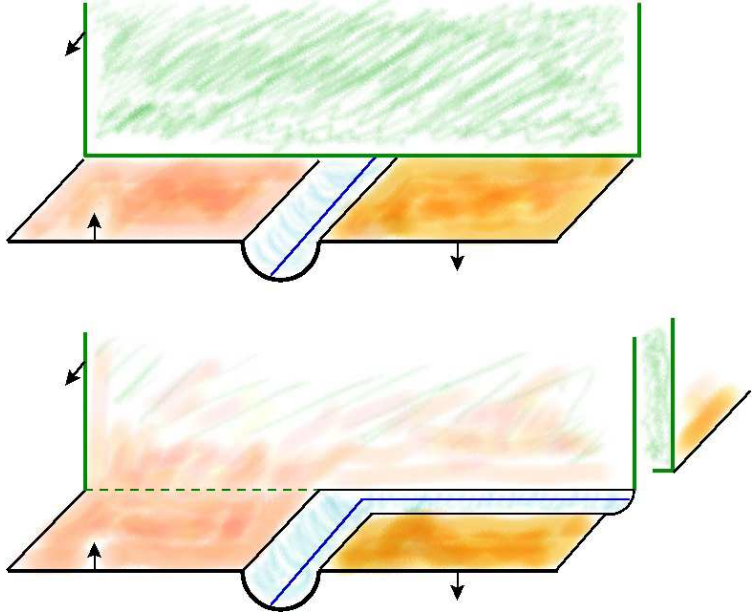}
\caption{The top figure shows the boundary of a sutured manifold (with orientations on $R_\pm(\gamma)$ indicated. The green surface is a decomposing surface. To decompose along the green surface, cut open the manifold along the surface and take the double curve sum of the sutures with the boundary of the decomposing surface. The bottom figure shows the decomposed sutured manifold.}
 \label{Fig:  SuturedDecomp}
\end{figure}

We write the decomposition as $(M,\gamma,\beta) \stackrel{S}{\to} (W,\gamma',\beta')$. The decomposition is \defn{$\beta$-taut} if $(M,\gamma,\beta)$ and $(W,\gamma',\beta')$ are $\beta$-taut and $\beta'$-taut, respectively. The decomposition \defn{respects} the parameterizing surface $Q$ if $Q'$ is a parameterizing surface. A sequence of sutured manifold decompositions is a \defn{hierarchy} of $(M,\gamma,\beta)$ if it concludes with a sutured manifold $(M_n,\gamma_n,\beta_n)$ such that $H_2(M_n, \boundary M_n) = 0$.  For more information on the type of surfaces along which we decompose our sutured manifold, see \cite[Definition 2.1]{S2}. We call any of the surfaces appearing in that definition an \defn{allowable decomposing} surface. 

One property possessed by allowable decomposing surfaces is that if $e$ is an oriented edge or circle of $\beta$, then if both the allowable decomposing surface and $R(\gamma)$ intersect $e$, they do so with the same sign. Let $(M_i, \gamma_i, \beta_i)$ be the sutured manifolds in a $\beta$-taut sutured manifold hierarchy using allowable decomposing surfaces and give each edge and loop of $\beta_i$ in $(M_i,\gamma_i,\beta_i)$ an orientation so that each edge of $\beta_i$ with both endpoints in $R(\gamma_i)$ is oriented from $R_-(\gamma_i)$ to $R_+(\gamma_i)$. Then in each of the sutured manifolds $(M_i,\gamma_i,\beta_i)$  the orientation induced on each edge $e$ of $\beta_i$  either always coincides the original orientation of corresponding edge or loop in $\beta$ or always opposes it.

The next three results are fundamental to passing information up and down sutured manifold hierarchies.

\begin{theorem}[{\cite[Theorem 2.5]{S2}}]\label{Theorem Down}
Every $\beta$-taut sutured manifold $(M,\gamma,\beta)$ admits a $\beta$-taut sutured manifold hierarchy respecting a given parameterizing surface using allowable decomposing surfaces.
\end{theorem}

\begin{theorem}[{\cite[Corollary 2.7]{S2}} and {\cite[Corollary 3.3]{S1}}]\label{Theorem Up}
Suppose that $\mc{H}$ is a $\beta$-taut sutured manifold hierarchy of $(M,\gamma,\beta)$ using allowable decomposing surfaces and that no component of $M$ is a solid torus disjoint from $\gamma$. Suppose that $\mc{H}$ terminates in $(M_n, \gamma_n, \beta_n)$. If $(M_n,\gamma_n,\beta_n)$ is $\nil$-taut, then every decomposition in the hierarchy is $\nil$-taut. In particular, $(M,\gamma,\nil)$ is $\nil$-taut. Furthermore, every decomposing surface is $\nil$-taut.
\end{theorem}

Finally,

\begin{theorem}[{\cite[Lemmas 7.5 and 7.6]{S1}}]\label{Index Decr.}
Suppose that $(M,\gamma,\beta) \stackrel{S}{\to} (W,\gamma',\beta')$ is a sutured manifold decomposition respecting a parameterizing surface $Q$. If $Q'$ is the resulting parameterizing surface in $M'$ then $I(Q) \geq I(Q')$.
\end{theorem}

This last result motivates the study of lower bounds for index at the end of the hierarchy. We take up that task in the next subsection.

\subsection{Combinatorics at the end}

Throughout this section, suppose that $(M,\gamma,\beta)$ is a $\beta$-taut sutured manifold with $\beta$ the union of arcs and $\boundary M \neq \nil$. Assume that $M$ is connected and that $H_2(M,\boundary M)  = 0$. This last assumption implies that $\boundary M$ is the union of spheres.

The next theorem is a reworking of the combinatorics from \cite[Section 9]{S1} which are in turn based on combinatorics found in \cite{G3}.

\begin{theorem}\label{Combinatorics}
Orient each edge of $\beta$ from $R_-(\gamma)$ to $R_+(\gamma)$. Suppose that $(M,\gamma,\nil)$ is not $\nil$-taut and that $Q \subset (M,\gamma,\beta)$ is a parameterizing surface traversing each edge of $\beta$ at least $\mu \geq 1$ times. Then one of the following occurs:
\begin{enumerate}
\item There is a Gabai disc for $Q$ in $(M,\gamma,\beta)$.
\item $I(Q) > 2\mu$
\item $(M,\gamma)$ satisfies all of the following:
\begin{enumerate}
\item $|\gamma| = 1$
\item $M$ is a rational homology ball
\item $H_1(M)$ has torsion
\item $M$ has Heegaard genus at most 3
\item If $M$ does not have a lens space summand, then $M$ has Heegaard genus at most 2.
\end{enumerate}
\end{enumerate}
\end{theorem}

\begin{proof}[Proof of Theorem \ref{Combinatorics}]
Assume that there is no Gabai disc for $Q$ in $(M,\gamma,\beta)$. 

\begin{lemma}\label{Elementary}
No component of $\boundary M$ is disjoint from $\beta$ and each disc component of $R(\gamma)$ contains an endpoint of $\beta$. Also, no component of $R(\gamma)$ is a sphere.
\end{lemma}
Since $(M,\gamma)$ is $\beta$-taut, but not $\nil$-taut, $\beta \neq \nil$. Since $M-\beta$ is irreducible, no component of $\boundary M$ is a sphere disjoint from $\beta$. If a component of $R(\gamma)$ were a sphere, then the complement of a regular neighborhood of a point in that component would be a Gabai disc for $Q$. Thus, no component of $R(\gamma)$ is a sphere.

If some component $R$ of $R_\pm$ were a disc disjoint from $\beta$, then the component $R'$ of $R_\mp$ adjacent to $R$ would either be $\beta$-compressible, or would be a disc disjoint from $\beta$. The former case contradicts the assumption that $(M,\gamma,\beta)$ is $\beta$-taut and the latter implies that a component of $\boundary M$ is a sphere disjoint from $\beta$, a contradiction.
\end{proof}

Certain disc components of $Q$ may be special from the point of view of sutured manifold theory. A disc component $q$ whose boundary traverses exactly one edge of $\beta$ and crosses exactly one suture is called a \defn{cancelling disc}. A disc component $q$ whose boundary traverses exactly two (distinct) edges of $\beta$ and no sutures is called a \defn{non-self amalgamating disc}. See Figure \ref{Fig:  cancelling} for a depiction. A disc whose boundary crosses $\gamma$ exactly twice and is disjoint from $\beta$ is a \defn{product disc}.  A disc whose boundary traverses a single arc exactly twice and is disjoint from $\gamma$ and all other arcs is a \defn{self amalgamating disc}. Cancelling discs, non-self amalgamating discs, product discs, and self-amalgamating discs are exactly the parameterizing surfaces of index 0. By \cite[Lemmas 4.3 and 4.4]{S1}, if $e$ is an edge of $\beta$ contained in the boundary of a cancelling or non-self amalgamating disc, the sutured manifold $(M,\gamma,\beta - e)$ is $(\beta-e)$-taut. We may view the removal of $e$ as an isotopy of $e$ across $q$ so that it either ends up in $\boundary M$ (in the case of a cancelling disc) or is merged into another edge of $\beta$ (in the case of a non-self amalgamating disc.) Any other component of $Q$ whose boundary runs along $e$ can be pushed along $q$ so that it either crosses a suture instead of traversing $e$ or traverses an arc of $\beta - e$. Sometimes there are methods for removing product discs and self-amalgamating discs, but we will not need them in this paper.

\begin{figure}[tbh]
\centering
\includegraphics[scale=0.4]{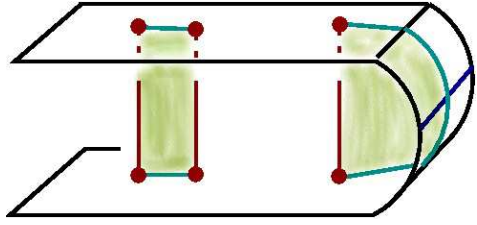}
\caption{A non-self amalgamating disc is on the left and a cancelling disc is on the right}
 \label{Fig:  cancelling}
\end{figure}

We remove components of $\beta$ using cancelling discs and non-self amalgamating discs, one after the other to obtain $\beta'$ and a parameterizing surface $Q'$ in $(M,\gamma,\beta')$. We note that $(M,\gamma,\beta')$ is $\beta'$-taut and we do this as many times as necessary to guarantee that no component of $Q'$ is a cancelling disc or non-self amalgamating disc. Note that $I(Q) = I(Q')$.

We consider the points $\beta \cap \boundary M$ as vertices of a graph $\Gamma$ and components $\boundary Q \cap \boundary M$ as its edges. By hypothesis, the valence of each vertex of $\Gamma$ is at least $\mu$. Similarly, we consider the points $\beta' \cap \boundary M$ as vertices of a graph $\Gamma'$ having edges $\boundary Q' \cap \boundary M$.  For a vertex $v \in \Gamma'$, let $\rho(v)$ denote the number of edges of $\Gamma'$ having exactly one endpoint at $v$ (so no loops) and say that $v$ is \defn{full} if $\rho(v) \geq \mu$. Similarly, for a component $\delta$ of $\boundary R_\pm(\gamma)$, let $\rho(\delta)$ denote the number of edges of $\Gamma' \cap R_\pm(\gamma)$ having exactly one endpoint at $\delta$. Say that $\delta$ is \defn{full} if $\rho(\delta) \geq \mu$. 

A \defn{loop} of $\Gamma'$ is an edge of $\Gamma'$ disjoint from $\gamma$ with both endpoints at the same vertex. A loop is \defn{inessential} if it bounds a disc in $\boundary M$ with interior disjoint from vertices of $\Gamma'$. Since no component of $R(\gamma)$ is a disc disjoint from the vertices of $\Gamma$, or for that matter, the vertices of $\Gamma'$ (Lemma \ref{Elementary}), the disc bounded by an inessential loop is disjoint from $\gamma$.

The next lemma uses the absence of Gabai discs to guarantee that certain vertices and components of $\boundary R(\gamma)$ are full.

\begin{lemma}\label{Apply Gabai Discs}
The following are full:
\begin{enumerate}
\item vertices $v \in \Gamma'$ at which no essential loop is based.
\item vertices $v \in \Gamma'$ at which a loop $\alpha$ is based such that $\alpha$ bounds a disc $D$ in $\boundary M$ having the property that all vertices of $\Gamma'$ interior to $D$  lie in either $R_-$ or in $R_+$.
\item Components $\delta$ of $\boundary R_\pm$ such that the disc $D$ in $\boundary M$ bounded by $\delta$ and containing the component of $R_\pm$ to which $\delta$ belongs has the property that all vertices of $\Gamma'$ interior to it lie in either $R_-$ or $R_+$.
\item Components $\delta$ of $\boundary R_\pm$ bounding disc components of $R_\pm$.
\end{enumerate}
\end{lemma}
\begin{proof}
Suppose, to begin, that $v$ is a vertex in $\Gamma'$ at which no essential loop is based. Let $D_v \subset \boundary M$ be a disc such that $v$ is the only vertex of $\Gamma'$ contained in $D_v$ and so that $D_v$ contains all the inessential loops based at $v$ and so that $|\boundary D_v \cap \Gamma'| = \rho(v).$   If an arc of $\beta$ is merged, during the creation of $\beta'$, with the component of $\beta$ having endpoint $v$, then we can expand $D_v$ to contain the endpoint of the amalgamated arc and the edge of the amalgamating disc joining it to $v$. Thus, $|\boundary D_v \cap \Gamma| = \rho(v)$. If $\rho(v) < \mu$, then $D_v$ would be a Gabai disc for $Q$, a contradiction. Hence, we have shown (1).

We now prove (2). Slightly inflate $D$ to a disc $D_v$ containing $\alpha$ in its interior and note that \[ |\boundary D_v \cap \Gamma| =  |\boundary D_v \cap \Gamma'| = \rho(v).\] Although, by hypothesis each point of $\Gamma' \cap D_v$ has the same sign, the same may not be true of $\Gamma \cap D_v$. Suppose, therefore, that $e'$ is an edge of $\beta$ that is cancelled into a suture contained in $D_v$. We may push the interior of $D_v$ into the interior of $M$, sliding it along the cancelling disc so that $D_v$ is disjoint from $e'$. Combining that operation with the inflation operation of the previous paragraph, we construct a Gabai disc for $Q$ unless $\rho(v) \geq \mu$.

The proof of (3) is almost identical to that of (2). If some edge $e'$ of $\beta$ is cancelled into $\delta$, we slide $D$ across the cancelling disc to make it disjoint from $e'$. Combined with the fact that every disc component of $R_\pm$ contains a point of $\beta$ and $\beta'$ and the inflation and pushing operations arising in the proofs of (1) and (2), we can construct a Gabai disc for $Q$ unless $\rho(\delta) \geq \mu$.

Conclusion (4) follows immediately from conclusion (3), for any component of $\boundary R_\pm$ bounding a disc component of $R_\pm$ satisfies the hypotheses of (3).
\end{proof}

The next lemma guarantees the existence of full vertices.

\begin{lemma}\label{Full vertices Exist}
The following hold:
\begin{enumerate}
\item Suppose that $R$ is a component of $R(\gamma)$ and that $\delta$ is a component of $\boundary R$. Let $E$ be the component of $\boundary M - \alpha$ containing $R$. Then there exists a full vertex in $E$. In particular, every disc component of $R_\pm$ contains a full vertex.
\item If $\delta$ is an essential loop, then there exists a full vertex (other than the vertex in $\delta$) on both sides of $\delta$.
\end{enumerate}
\end{lemma}
\begin{proof}
If $\delta$ is a component of $\boundary R_\pm$, there exists a disc component $D$ of $R(\gamma)$ contained in $E$, so it suffices to prove the lemma in the cases when $\delta$ is either the boundary of a disc component of $R_\pm$ or an innermost essential loop. By Lemma \ref{Elementary}, every disc component of $R_\pm$ contains a vertex of $\Gamma'$. If there is no essential loop based at the vertex, then it is full, so suppose that there is an essential loop. Without loss of generality, we may assume it is innermost in $D$. Since $D$ is a disc, there is a vertex $v$ interior to that essential loop. No essential loop can be based at $v$ and so $v$ is full. Thus, we have shown (1) and we have also shown that if $\delta$ is an essential loop in a disc component of $R(\gamma)$ then there is a full vertex interior to the loop. To conclude, note that if $\delta$ is any essential loop then there is a suture or vertex on each side of $\delta$. Let $S$ be one of the sides. If there is a suture in $S$, then there is a disc component of $R(\gamma)$ in $S$ and we are done. If there is no suture in $S$, then both $\delta$ and $S$ must be contained in a disc component of $R(\gamma)$ and we are, once again, done.
\end{proof}

The next lemma produces the inequality $I(Q) > 2\mu$. It does not use the assumption about the non-existence of Gabai discs.

\begin{lemma}\label{Connectivity}
Suppose that $v,w \in R_\pm$ are full vertices of $\Gamma'$ and that $\delta$ is a full component of $\boundary R_\pm$. Then either
\begin{enumerate}
\item[(a)] there is an edge of $\Gamma' \cap R_\pm$ joining $v$ to $w$ and joining each of $v$ and $w$ to $\delta$, or
\item[(b)] $I(Q) \geq 2\mu$.
\end{enumerate}  
\end{lemma}

If conclusion (a) holds for $v$ and $w$ we say that they are \defn{adjacent} vertices of $\Gamma' \cap R_\pm$ and if conclusion (a) holds for $v$ and $\delta$ we say that $v$ and $\delta$ are \defn{adjacent} in $R_\pm$.

\begin{proof}
Suppose that $v$ and $w$ are non-adjacent full vertices in $\Gamma' \cap R_\pm$. Let $\epsilon$ be a component of $\boundary Q'$ that traverses the edge $e_v$ of $\beta'$ with endpoint $v$ exactly $s$ times and traverses the edge of $e_w$ of $\beta'$ with endpoint $w$ exactly $t$ times. Assume that $s + t \geq 1$. Let $L_\epsilon$ denote the number of loops of $\epsilon \cap \boundary M_n$ based at either $v$ or $w$. Since $v$ and $w$ are non-adjacent, no component of $\epsilon \cap \boundary M_n$ joins $v$ to $w$ without passing through a suture. Thus, number of times that $\epsilon$ passes through the sutures or traverses an edge of $\beta' - (e_v \cup e_w)$ is at least $(s+t) - 2L_\epsilon$. 

Let $q_\epsilon$ denote the component of $Q'$ containing $\epsilon$. If $q_\epsilon$ is a disc, we have
\[
I(q_\epsilon) \geq -2 + 2(s+t) - 2L_\epsilon
\]
Moreover, for $q_\epsilon$ a disc, $I(q_\epsilon)$ is at least $s+t - 2L_\epsilon$ unless $s+t = 1$. The latter possibility implies that $q_\epsilon$ is a cancelling disc, which is a contradiction to $q_\epsilon$ being a component of $Q'$. Thus, if $q_\epsilon$ is a disc, then $I(q_\epsilon) \geq s+ t - L_\epsilon$.

Suppose that $q_\epsilon$ is not a disc. Let $\epsilon_1, \hdots, \epsilon_p$ be the components of $\boundary q_\epsilon$ and let $s_i$ and $t_i$ be the number of times that $\epsilon_i$ traverses $e_v$ and $e_w$ respectively. Let $L_i$ be the number of loops of $\epsilon_i$ based at either $v$ or $w$. Since $q_\epsilon$ is not a disc, we have
\[
I(q_\epsilon) \geq \sum_{i=1}^p 2(s_i + t_i) - 2L_i \geq \sum_{i=1}^p (s_i - t_i) - 2L_i
\]
Summing over all components of $Q'$ that traverse either $e_v$ or $e_w$ we find that $I(Q)$ is at least the number of times that $\boundary Q$ traverses $e_v \cup e_w$ minus twice the number of loops based at $v$ and $w$. That is, $I(Q) \geq \rho(v) + \rho(w) \geq 2\mu$.

To prove the lemma for $v$ and $\delta$ and for $w$ and $\delta$, we could repeat the previous proof with $\delta$ in place of $w$ or apply what we have already done to the sutured manifold obtained by converting each arc of $\beta$ to a suture.
\end{proof}

We can now deduce that either $I(Q) \geq 2\mu$ or $|\gamma| = 1$.

\begin{lemma}\label{Complete Graph}
Assume that $I(Q) < 2 \mu$. Then the boundary of $M$ consists of a single sphere and $|\gamma| = 1$. Furthermore, there are no essential loops in $\Gamma'$ and every vertex and each component of $\boundary R(\gamma)$ is full. Consequently, $\Gamma' \cap R_\pm(\gamma)$ is the union of simple closed curves and a complete planar graph where each vertex is joined to $\boundary R_\pm(\gamma)$ by an edge.  
\end{lemma}
\begin{proof}
By Lemma \ref{Elementary}, no component of $R(\gamma)$ is a sphere. Suppose that $D_1$ and $D_2$ are disc components of $R_\pm$. Each of $D_1$ and $D_2$ contains a full vertex by Lemma \ref{Full vertices Exist}. By Lemma \ref{Connectivity}, those full vertices must be connected by an edge of $\Gamma' \cap R_\pm$. Thus, $D_1 = D_2$. We conclude that $R_\pm$ has at most one component that is a disc. If a component $P$ of $\boundary M$ contains a suture, then at least two components of $R(\gamma) \cap P$ are discs. Thus, $\boundary M$ is connected and there is exactly one disc component $D_-$ of $R_-$ and exactly one disc component $D_+$ of $R_+$. Hence, on $\boundary M$ all sutures are parallel. We recall that both $D_-$ and $D_+$ contain a full vertex. Indeed by Lemma \ref{Connectivity}, they contain every full vertex of $\Gamma' \cap R_-$ and $\Gamma' \cap R_+$ respectively. 

Suppose that $A$ is a component of $R_\pm - D_\pm$. Since $D_\pm$ contains a full vertex and since full vertices in $R_\pm$ are adjacent, $A$ cannot contain a full vertex. Suppose that $v$ is a vertex of $\Gamma' \cap A$. Since $v$ is not full, by Lemma \ref{Apply Gabai Discs}, there is an essential loop $\epsilon$ based at $v$. By Lemma \ref{Full vertices Exist}, there is a full vertex $w$ interior  to $\epsilon$. Since full vertices in $R_\pm$ are adjacent, $w$ must lie in either $D_\pm$ or $D_\mp$. Consequently, $D_\pm$ or $D_\mp$ is interior to $\epsilon$. We conclude that, as simple closed curves on $\boundary M$, all sutures and all essential loops are parallel.  Number, in order, the annuli of $R(\gamma)$ by $A_1, \hdots, A_n$ so that $A_1$ is adjacent to $D_-$ and $A_n$ is adjacent to $D_+$. Note that the $A_i$ alternate between $R_+$ and $R_-$ beginning with $R_+$. 

Suppose that $i$ is the smallest number such that $A_i$ contains a vertex of $\Gamma'$. Since there is an essential loop at every vertex of $\Gamma' \cap A_i$ and since each essential loop is parallel to the sutures, we may choose such a loop $\epsilon$, based at a vertex $v \in A_i \cap \Gamma'$ that is closest to $D_-$ and so any vertex in the component of $\boundary M - \epsilon$ containing $D_-$ is contained in $D_-$. Then $v$ satisfies the conditions of Conclusion (2) of Lemma \ref{Apply Gabai Discs}. Hence $v$ is full and must, therefore, lie in $D_-$ or $D_+$. We conclude that every vertex of $\Gamma'$ lies in $D_\pm$.

Since each vertex of $\Gamma'$ lies in $D_\pm$, by Conclusion (3) of Lemma \ref{Apply Gabai Discs}, each component of $\boundary R(\gamma)$ is full. By Lemma \ref{Connectivity}, each boundary component is joined to a vertex of $D_\pm$ by an edge of $\Gamma' \cap R_\pm$. Consequently, each component of $\boundary R(\gamma)$ is a component of $\boundary D_\pm$ and so $|\gamma| = 1$. 

Finally, notice that if $\epsilon$ is an essential loop in $\Gamma'$, then a full vertex in its interior could not be joined by an edge of $\Gamma' \cap R(\gamma)$ to $\boundary R(\gamma)$. Thus there are no essential loops in $\Gamma'$. Since there are no essential loops, by Lemma \ref{Apply Gabai Discs}, each vertex is full.
\end{proof}

\begin{lemma}\label{Heeg. genus 3}
$M$ is a rational homology sphere with torsion in $H_1(M;\Z)$ and $M$ has Heegaard genus at most 3.
\end{lemma}
\begin{proof}
By the easy half of Kuratowski's theorem, $\Gamma' \cap R_\pm$ has at most 3 vertices, so $\beta$ has at most 3 edges. Let $\beta_i$ be one such edge, let $v$ be one of its vertices, and let $\st(v)$ denote the edges of $\Gamma' \cap R_\pm$ with exactly one endpoint at $v$. Let $q$ be a component of $Q'$ whose boundary traverses $\beta_i$. Then,
\[
I(q) \geq -2\chi(q) + 2|\boundary q \cap \st(v)|.
\]
If $q$ is not a disc, then $I(q) \geq 2|\boundary q \cap \st(v)|$. Recall that $I(q)$ is even. Thus, if $q$ is a disc and if the inequality is strict, then once again we have $I(q) \geq 2|\boundary q \cap \st(v)|$. If no component of $Q'$ traversing $\beta_i$ is a disc $q$ satisfying $I(q) = -2 + 2|\boundary q \cap \st(v)|$, then summing over all components $q$ of $Q'$ traversing $\beta_i$, we would have:
\[
2\mu > I(Q') \geq 2\rho(v) \geq 2\mu,
\]
a contradiction. Consequently, some component $q_i$ of $Q'$ is a disc with the property that each component of $\boundary q \cap R_\pm$ lies in $\st(v)$. Consequently, $\boundary q$ alternately traverses $\beta_i$ and then either crosses $\gamma$ or traverses an arc of $\beta' - \beta_i$. We call such a disc an \defn{alternating disc}.

Let $V$ be the union of a collar neighborhood of $\boundary M$ with a regular neighborhood of $\beta'$. Note that $V$ is a compressionbody with $\boundary_- V = \boundary M$ and $\boundary_+ V$ a surface of genus at most 3. Since the boundary of each $q_i$ traverses the corresponding $\beta_i$ always in the same direction, each $q_i$ is non-separating. Since no $q_i$ is a cancelling or amalgamating disc, $\boundary q_i$ traverses $\beta_i$ at least twice. See Figure \ref{Fig:  CompBody} for a picture of $V$.

\begin{figure}[tbh]
\centering
\includegraphics[scale=0.3]{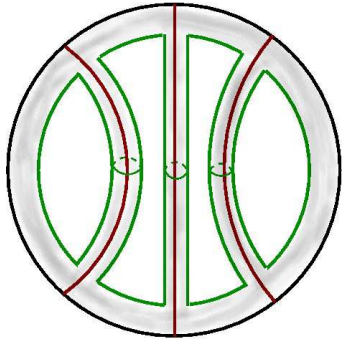}
\caption{The compressionbody $V$ is shaded.}
 \label{Fig:  CompBody}
\end{figure}

Compressing $\boundary_+ V$ along the discs $q_i$ for each $i$ produces the union of spheres in $M$ each of which is disjoint from $\beta'$. Since $(M,\gamma,\beta')$ is $\beta'$-taut, each of those spheres bounds a 3-ball in $M - \beta'$. Consequently $M - \inter{V}$ is a handlebody. Thus, $\boundary_+ V$ is a Heegaard surface for $M$ of genus at most 3. By \cite[Lemma 9.6]{S1}, $M = M' \# W$ where $H_1(W)$ is finite and non-trivial.
\end{proof}

Finally, we show that there are very strong restrictions on the 3-manifold $M$.
\begin{lemma}
$M$ has a connect summand of Heegaard genus one or two.
\end{lemma}
\begin{proof}
By Lemma \ref{Heeg. genus 3}, $M$ has Heegaard genus at most 3. If $M$ is reducible, then by Haken's Lemma it must have a connect summand that is a lens space. Thus, we may assume that $M$ is irreducible.

Let $\wihat{M}$ be the result of attaching a 2-handle along $\gamma$ and let $\beta_0$ be the cocore of the 2-handle. Consider the sutured manifold $(\wihat{M}, \nil, \beta' \cup \beta_0)$. Note that $Q'$ is still a parameterizing surface and that it has the same index. Let $\Gamma''$ be the graph of intersection between $\boundary Q'$ and $\boundary \wihat{M}$. Since each component of $\boundary R_\pm(\gamma)$ was full, each vertex of $\Gamma'$ is full and each component of $\Gamma''$ is a complete graph planar graph with at most 4 vertices. 

Discard from $Q'$ any component not traversing an edge of $\beta' \cup \beta_0$, let $Q'_1$ be the resulting surface. Since the index of each component is always non-negative, $I(Q') \geq I(Q'_1)$. Any component of $\boundary Q'_1$ that is a simple closed curve (disjoint from the vertices of $\Gamma''$) bounds a disc in the complement of $\Gamma''$ by Lemma \ref{Complete Graph}. Beginning with innermost discs attach discs to such components of $\boundary Q'_1$ and push into $\wihat{M}$ so that $Q'$ remains properly embedded. Let $Q'_2$ be the resulting surface. Since each component of $Q'_1$ had boundary traversing $\beta' \cup \beta_0$, $Q'_2$ remains a parameterizing surface and $I(Q'_1) \geq I(Q'_2)$. Discard from $Q'_2$ any component $q$ whose boundary has the property that for each component of $\beta' \cup \beta_0$, $\boundary q$ traverses the component the same number of times in each direction. That is, $\boundary q \cap \boundary \wihat{M}$ consists only of loops (necessarily inessential). Let $Q'_3$ be the resulting surface and observe that $I(Q'_2) \geq I(Q'_3)$. Finally, in the exterior of $\beta' \cup \beta'_0$, isotope $Q'_3$ to a surface $Q'_4$ by eliminating all the loops (which must be inessential) of $\boundary Q'_3 \cap \boundary \wihat{M}$. Observe that $Q'_4$ is still a parameterizing surface and that $I(Q'_3) \geq I(Q'_4)$. Let $\Gamma''_4$ be the graph of intersection between $\boundary Q'_4$ and $\boundary M$. Note that the only edges (not circles) that have been eliminated from $\Gamma''$ to obtain $\Gamma''_4$ are those that are inessential loops. As before, remove from $Q'_4$ any component disjoint from $\beta' \cup \beta_0$ and then cap off any circle components of $Q'_4 \cap \boundary \wihat{M}$. Let $Q'_5$ be the resulting surface and let $\Gamma'_5$ be its graph of intersection with $\wihat{M}$. Observe that $Q'_5$ is a parameterizing surface, that $I(Q'_4) \geq I(Q'_5)$, that there are no loops or circles in $\Gamma'_5$, and that each vertex of $\Gamma'_5$ remains full. 

Recall that $V$ is the compressionbody that is the union of a collar neighborhood of $\wihat{M}$ with a regular neighborhood of $\beta'_1 = \beta' \cup \beta'_0$. We showed in the proof of Lemma \ref{Heeg. genus 3} that the exterior of $V$ in $\wihat{M}$ (equivalently, in $M$) is a handlebody. If some component of $Q'_5$ is a non-self amalgamating disc, then the surface $\boundary_+ V$ is a stabilized Heegaard surface for $M$. Since $\boundary_+ V$ had genus at most 3, this implies that $M$ has genus at most 2. If $M$ had genus 0, it would be a 3-ball contradicting the fact that $H_1(M)$ has torsion (or that $(M,\gamma,\nil)$ is not $\nil$-taut). If $M$ had Heegaard genus 1, it would be a lens space. Thus if some component of $Q'_5$ is a non-self amalgamating disc, then $M$ has Heegaard genus 2 or is a lens space.

Suppose, therefore, that no component of $Q'_5$ is a non-self amalgamating disc. Since there are no loops in $\Gamma'_5$, each disc component of $Q'_5$ is an alternating disc for some component of $\beta'$. If there is such an alternating disc $q$ that traverses exactly two arcs of $\beta'_1$, then $M$ would have a lens space summand. Suppose that this does not occur. Then each disc of $Q'_5$ must traverse the arcs of $\beta'_1$ at least 4 times. Let $V$ be the total valence of the vertices of $\Gamma'_5$. The total number of discs is at most $V/4$. Thus,
\[
2\mu > I(Q) \geq V - 2V/4 = V/2.
\]
We have $V \geq \mu |\beta'_1|$, since each vertex of $\beta'_1$ is full. Consequently, $4 > |\beta'_1|$. Since $|\beta'_1|$ is an integer, $|\beta'_1| \leq 3$ and so $|\beta| \leq 2$. Since there is an alternating disc for each component of $\beta$, the Heegaard genus of $M$ is at most 2.
\end{proof}
And so we conclude the proof of Theorem \ref{Combinatorics}.

\begin{remark}
The assumption that $(M,\gamma,\nil)$ is not $\nil$-taut is used only to guarantee that the graph $\beta' \neq \nil$ and that if $|\gamma| = 1$, then $M$ is a non-trivial rational homology ball. If we allow $(M,\gamma,\nil)$ to be $\nil$-taut, then $M$ must be a 3-ball with a single suture in its boundary and we can conclude that either there is a Gabai disc for $Q$, or $I(Q) \geq 2\mu$, or $\beta' = \nil$. If $\beta' = \nil$ this implies that 1--manifold $\beta$ is unknotted in the 3-ball $M$.
\end{remark}

\subsection{The sutured manifold theorem}
In this section, suppose that $(N,\wihat{\gamma},\nil)$ is a connected $\nil$-taut sutured manifold and that $M$ is obtained from $N$ either by Dehn filling a component $T \subset T(\gamma)$ with slope $b$ or that $M$ is obtained from $N$ by attaching a 2-handle to a component $b$ of $\gamma$.  Let $\beta$ be either the core of the filling torus or the cocores of the 2-handle in $M$. Suppose that $Q \subset M$ is a parameterizing surface which is not disjoint from $b$.
\begin{theorem}\label{Sutured Mfld Thm}
Assume the following:
\begin{enumerate}
\item $H_2(M,\boundary M) \neq 0$
\item $(M,\gamma,\beta)$ is $\beta$-taut, where $\gamma = \wihat{\gamma} - b$
\item If $M$ is a solid torus then $\gamma \neq \nil$
\item Either $(M,\gamma,\nil)$ is not $\nil$-taut or $(M,\beta)$ has an exceptional class
\item $\boundary Q \cap b \neq \nil$ and no component of $Q$ is a disc or sphere disjoint from $\gamma \cup \beta$.
\end{enumerate}
Then one of the following holds:
\begin{enumerate}
\item There is a Gabai disc for $Q$ in $M$.
\item $I(Q) \geq 2|\boundary Q \cap b|$
\item $M = W_0 \# W_1$ where $W_0 \neq S^3$ and $W_1$ is a rational homology ball having non-trivial torsion in $H_1(W_2)$ of Heegaard genus at most 2.
\end{enumerate}
\end{theorem}

\begin{proof}
Let $\mu = |\boundary Q \cap b|$. Assume that $I(Q) < 2\mu$ and that there is no Gabai disc for $Q$ in $M$. 

If there is an exceptional class $\sigma$, let $S\subset M$ be a $\beta$-taut surface representing $\sigma$. Since $S$ is $\beta$-taut, it is not $\nil$-taut. Therefore, the sutured manifold decomposition
\[
(M,\gamma,\beta) \stackrel{S}{\to} (M',\gamma',\beta')
\]
given by $S$ is not $\nil$-taut \cite[Corollary 3.3]{S1}. By \cite[Theorem 2.6, Lemma 7.5]{S1} we may choose $S$ so that the decomposition is $\beta$-taut. If $(M,\gamma,\nil)$ is not taut, let $(M',\gamma',\beta') = (M,\gamma,\beta)$ and take $S = \nil$.

Let 
\[
(M',\gamma',\beta') \stackrel{S_1}{\to} (M_1, \gamma_1, \beta_1) \stackrel{S_2}{\to} \hdots \stackrel{S_n}{\to} (M_n,\gamma_n,\beta_n)
\]
be a $\beta'$-taut sutured manifold hierarchy of $(M',\gamma',\beta')$ respecting $Q'$ (Lemma \ref{Theorem Down}). The sequence
\[
\mc{H}: (M,\gamma,\beta) \stackrel{S}{\to} (M',\gamma',\beta') \stackrel{S_1}{\to} \hdots \stackrel{S_n}{\to} (M_n,\gamma_n,\beta_n)
\]
is then a $\beta$-taut sutured manifold hierarchy of $(M,\gamma,\beta)$ respecting $Q$. Since the first decomposition is not $\nil$-taut, $\mc{H}$ is not a $\nil$-taut sequence of sutured manifold decompositions. By Theorem \ref{Theorem Up}, $(M_n,\gamma_n, \nil)$ is not $\nil$-taut. Let $Q_n$ be the parameterizing surface in $M_n$ resulting from the decomposition of $Q$ \cite[Section 7]{S1}. By Theorem \ref{Index Decr.}, $I(Q) \geq I(Q_n)$. Since $\mc{H}$ is a $\beta$-taut hierarchy, $H_2(M_n, \boundary M_n) = 0$. This implies that $\boundary M_n$ is the union of spheres. Let $Q_n$ be the resulting parameterizing surface in $M_n$.

If $\beta_n$ is the knot $\beta$, then either $N$ is reducible (with a component of $\boundary M_n$ a reducing sphere) or $H_2(M,\boundary M) = 0$. Either possibility contradicts our initial hypotheses. Thus, $\beta_n$ is the union of arcs. The orientation of $\beta_n$ induced by the orientation of $\beta$ either makes each arc in $\beta_n$ run from $R_-(\gamma_n)$ to $R_+(\gamma_n)$ or makes each arc in $\beta_n$ run from $R_+(\gamma_n)$ to $R_-(\gamma_n)$. Without loss of generality, we may assume the former. Let $M'_n$ be a component of $M_n$ containing at least one arc of $\beta_n$. Let $Q'_n = Q_n \cap M'_n$ and $\gamma'_n = \gamma_n \cap M'_n$. By Theorem \ref{Combinatorics}, one of the following holds:
\begin{enumerate}
\item There is a Gabai disc for $Q'_n$ in $(M'_n,\gamma'_n,\beta'_n)$.
\item $I(Q'_n) \geq 2\mu$
\item $(M'_n,\gamma'_n)$ satisfies all of the following:
\begin{enumerate}
\item $|\gamma'| = 1$
\item $M'_n$ is a rational homology ball
\item $H_1(M'_n)$ has torsion
\item $M'_n$ has Heegaard genus at most 3
\item If $M'_n$ does not have a lens space summand, then $M'_n$ has Heegaard genus at most 2.
\end{enumerate}
\end{enumerate}

Since $\boundary M'_n$ is a sphere, it is a summand of $M$. Since $M - M'_n$ contains both $\boundary M$ and the first surface in the hierarchy (which if $\boundary M \neq \nil$ is closed and non-separating in $M$) the complement of $M'_n$ in $M$ is not a 3-ball. Thus $M'_n$ is a proper connect summand of $M$. Any Gabai disc for $Q'_n$ is also a Gabai disc for $Q$ and $I(Q) \geq I(Q'_n)$. And, thus, we have our theorem.
\end{proof}

\section{Gabai Discs}\label{Gabai Discs}

Motivated by Theorem \ref{Sutured Mfld Thm}, we seek to eliminate Gabai discs for the parameterizing surface $Q = \ob{Q} \cap N$ where $\ob{Q}$ is a surface in $M'$ transverse to $\alpha$. When $M$ and $M'$ are obtained from $N$ by Dehn-filling, it is well-known how to modify $\ob{Q}$, but for completeness we recap the basics here.  

Throughout this section, we assume that $M'$ is a compact, orientable 3-manifold containing a knot or arc $\alpha$ transverse to a  properly embedded orientable surface $\ob{Q}$.

\subsection{Rational homology cobordism}\label{rational homology cobord}
In this section we use Gabai discs to construct rational homology cobordisms between surfaces. Rational cobordisms are related to the notions of I-cobordism and J-cobordism introduced by Gabai and Scharlemann, respectively (see \cite[Appendix]{L2}), but we do not need those ideas here.

The following lemma will come in useful:
\begin{lemma}\label{Alg. Top.}
Assume that $\ob{Q}$ is a connected oriented surface transverse to $\alpha$. Let $\kappa$ is an arc component of $\alpha - \ob{Q}$ with endpoints on the same side of $\ob{Q}$. Suppose that when a tube $T$ is added to $\ob{Q}$ along $\kappa$ to produce a surface $\mathbf{H}$ there is a compressing disc $D$  for $\mathbf{H}$ outside $T$ such that $\boundary D$ crosses through $T$ exactly $q \geq 1$ times, always in the same direction. Let $\ob{R}$ be the result of compressing $\mathbf{H}$ using $D$. Then the following hold:
\begin{enumerate}
\item $|\ob{R} \cap \alpha| = |\ob{Q} \cap \alpha| - 2$
\item $\ob{R}$ is connected and $\ob{Q}$ and $\ob{R}$ have the same genus.
\item $\ob{Q}$ and $\ob{R}$ bound a 3-dimensional submanifold $W$ for which $\mathbf{H}$ is a Heegaard surface
\item $W$ is a rational cobordism. It is a product if and only if $q = 1$.
\item Unless $\ob{Q}$ is a sphere, $W$ is reducible if and only if it contains a lens space summand.
\item If $q \geq 2$ and if $\ob{Q}$ is a disc or sphere, then $W$ contains a lens space summand.
\item $\ob{Q}$ and $\ob{R}$ are incompressible in $W$.
\end{enumerate}
\end{lemma}

See Figure \ref{Fig:  ScharlCycle} for a depiction of the surfaces $\ob{Q}$, $\ob{R}$ and $\mathbf{H}$.

\begin{figure}[tbh]
\centering
\includegraphics[scale=0.4]{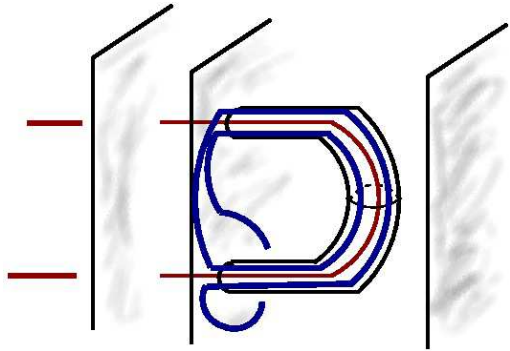}
\caption{The left surface is $\ob{Q}$; the right surface is $\ob{R}$; the middle surface is $\mathbf{H}$. The blue curve on $\mathbf{H}$ is $\boundary D$.}
 \label{Fig:  ScharlCycle}
\end{figure}

\begin{proof}
Begin by observing that $\ob{Q}$ can be constructed from $\ob{R}$ by attaching a 1-handle dual to $D$ and then attaching the 2-handle that is a meridian of $T$. The intersection number between the boundary of the 2-handle and the meridian of the 1-handle is equal to $q$. Thus, the picture is completely symmetric with respect to $\ob{Q}$ and $\ob{R}$.

By construction (1) and (3) are true, noting that if $\boundary \ob{Q} \neq \nil$ then $\mathbf{H}$ will be a Heegaard surface with boundary (also called a \defn{relative Heegaard surface} \cite{CM}). Since $\boundary D$ always passes through $T$ in the same direction, $\boundary D$ is non-separating on $\mathbf{H}$. Hence, $\ob{R}$ is connected. The genus of $\mathbf{H}$ is one greater than the genus of $\ob{Q}$ and so $\ob{Q}$ and $\ob{R}$ have equal genus. Thus, (2) holds.

Suppose that $W$ is a product. Let $\wihat{W}$ be the result of attaching 2-handles to $\boundary \mathbf{H}$ in $W$ and let $\wihat{\mathbf{H}} \subset \wihat{W}$ be the result of capping $\boundary \mathbf{H}$ off with discs. Since $\wihat{\mathbf{H}}$ is a Heegaard surface that separates the components of $\boundary \wihat{W}$ and since it has genus one greater than the genus of $\ob{Q}$, by \cite{ST} it is a stabilized Heegaard surface. Indeed, we note that the cocores $h$ of the 2-handles are vertical arcs in $\wihat{W}$ and that $\wihat{\mathbf{H}}$ intersects each of them in a single point. By \cite{HS-class}, $\wihat{\mathbf{H}}$ is a stabilized bridge splitting of $(\wihat{W}, h)$. Let $P$ be a sphere in $\wihat{W}$ intersecting $\wihat{\mathbf{H}}$ in a single essential loop, disjoint from $h$. Since $\mathbf{H}$ is obtained from $\ob{Q}$ by attaching a single 1-handle and since $\ob{R}$ is obtained from $\mathbf{H}$ by attaching a single 2-handle, the sphere $P$ can be isotoped (transversally to $\mathbf{H}$) so as to be disjoint from both the tube $T$ and the disc $D$. Inside $P$, the disc $D$ is a compressing disc for a once-punctured torus containing $D$. Since $P$ bounds a 3-ball, $\boundary D$ must pass through $T$ exactly once. Conversely, if $q = 1$, it is easy to see that $W$ is a product since $\wihat{H}$ is stabilized. Thus we have the second part of (3).

If $W$ is reducible, we apply the argument of the previous paragraph (using Hayashi and Shimokawa's version of Haken's lemma \cite[Theorem 1.3]{HS} in the case when $\boundary \ob{Q} \neq \nil$) to show that there is a sphere $P$ bounding a lens space summand of $W$. This is conclusion (4).

If $\ob{Q}$ is a disc or sphere, all arcs of $\boundary D \cap (\mathbf{H} - T)$ are parallel. Thus, $D$ is a compressing disc for a once-punctured torus and so, if $q \geq 2$ we have a lens space summand of $W$. This is conclusion (5).

Suppose that $\ob{Q}$ is compressible in $W$.  By \cite{CG} there is an essential disc $P$ in $W$ intersecting $\mathbf{H}$ in a single loop. (If $\boundary \ob{Q} \neq \nil$, attach 2-handles to $\boundary W$ as before and use \cite[Theorem 1.3]{HS}.) Let $P^*$ be the disc component of $P - \mathbf{H}$. We may isotope $P$ so that $P^*$ is disjoint from $D$. If $\boundary P^*$ separated $\mathbf{H}$, then it would bound a disc in the component of $M - \mathbf{H}$ containing $\ob{Q}$. This would contradict the fact that $\boundary P$ is essential on $\ob{Q}$. If $\boundary P^*$ does not separate $\mathbf{H}$, the disc $P^*$ must be parallel to $D$. But, then $P - P^*$ is an annulus extending $\boundary D$ to $\ob{Q}$, which contradicts the assumption that $\boundary D$ runs through the tube $T$. Thus, $\ob{Q}$ is incompressible in $W$. A dual argument shows that $\ob{H}$ is also incompressible in $W$. This is conclusion (6).

It remains only to show that $W$ is a rational homology cobordism. For simplicity, we prove this only in the case when $\boundary \ob{Q} = \nil$. The general case can be deduced from this one by attaching 2-handles to $\boundary \ob{Q}$ as above. For what follows, we take all homology groups to have $\Q$ coefficients.

Let $V_Q$ and $V_R$ be the closures of the components of $W - \mathbf{H}$ containing $\ob{Q}$ and $\ob{R}$ respectively. An easy argument with the Mayer-Vietoris sequence shows that the inclusion of $\ob{Q}$ into $W$ induces an isomorphism from $H_2(\ob{Q})$ to $H_2(W)$. We concentrate on showing that the inclusion also induces an isomorphism on first homology. Let $\{x_1, \hdots, x_{2g}\}$ be a basis for $H_1(\ob{Q})$ consisting of essential embedded oriented simple closed curves, each disjoint from the attaching discs for $T$. Let $\{m,l\}$ be an oriented meridian of $T$ and $l$ an oriented simple closed curve in $\mathbf{H}$ intersecting $m$ once and disjoint from the curves $x_1, \hdots, x_{2g}$. We may choose $l$ so that it has the same sign of intersection with $m$ as does $\boundary D$. The curve $\boundary D$ is homologous to $w = a_1x_1 + \hdots a_{2g}x_{2g} + pm + ql$ for some $a_1, \hdots, a_{2g}, p \in \Q$. Note that $q > 0$ is also the geometric intersection number between $\boundary D$ and $m$. 

The compressionbody $V$ deformation retracts to the union of $\ob{Q}$ with the cocore of the 1--handle bounded by $T$. We may take the images of $x_1, \hdots, x_{2g}$ and $l$ as a basis for $H_1(V)$, and we continue to refer to them as $x_1, \hdots, x_{2g}, l$. The 3-manifold $W$ deformation retracts to the union $C$ of $\ob{Q} \cup D$ with the core $\kappa$ of the 1-handle. We apply the Mayer-Vietoris Sequence to $C$ applied to $v = \ob{Q} \cup \kappa$ and $D$. We let $\delta$ be the image of $\boundary D$ in $v$. The Mayer-Vietoris sequence gives:
\[
0 \to H_1(\delta) \to H_1(v) \oplus H_1(D) \stackrel{\psi}{\to} H_1(C) \to 0.
\]
The kernel of $\psi$ is exactly the image of $w$ in $H_1(v)$, and so $H_1(C) = \langle x_1, \hdots, x_{2g}, l : a_1x_1 + \hdots a_{2g}x_{2g} + ql \rangle$. Since $q > 0$ and since we are using rational coefficients, the images of $x_1, \hdots, x_{2g}$ are a basis for $H_1(C)$, as desired.
\end{proof}

\begin{remark}
We will make use of the construction given in Lemma \ref{Alg. Top.} to rule out Scharlemann cycles in certain intersection graphs. These techniques are well known and have been applied in many papers. For example, the case when $\ob{Q}$ is a sphere or disc first arises in Scharlemann's seminal papers \cite{S-smooth, S-unknotting}. The case when $\ob{Q}$ is a torus is considered in \cite[Lemma 2.2]{BZ}, and elsewhere. 
\end{remark}

\subsection{Scharlemann cycles and rational homology cobordism}\label{Sec: Construction}

Recall that we are considering a compact, orientable, irreducible 3-manifold $N$ having oriented curves $a$ and $b$ in its boundary  intersecting minimally up to isotopy such that the geometric intersection number $\Delta$ between $a$ and $b$ is at least 1. If $a$ and $b$ lie on a torus component of $N$, we Dehn fill that component to obtain manifolds $M'$ and $M$ (respectively) containing knots $\alpha$ and $\beta$, respectively. If $a$ and $b$ lie on a component of $\boundary N$ of genus at least 2, we attach 2-handles along $a$ and $b$ to obtain $M'$ and $M$ (respectively) and embedded arcs $\alpha$ and $\beta$ (respectively).  Let $\ob{Q} \subset M'$ be a compact orientable surface transverse to $\alpha$ and let $Q = \ob{Q} \cap N$. Suppose that $\ob{P} \subset M$ is a compact, orientable surface transverse to $\beta$ and let $P = \ob{P} \cap N$. We may isotope $P$ and $Q$ relative to their boundaries so that they are transverse. The intersection between $P$ and $Q$ is a 1-manifold. We consider the boundary components of $P$ and $Q$ that are parallel to $b$ and $a$ (in $\boundary M$) as (fat) vertices in $\ob{P}$ and $\ob{Q}$ respectively and the components of $P \cap Q$ as edges (some of which may be loops disjoint from the vertices). These vertices and edges we consider as graphs $\Gamma_P$ and $\Gamma_Q$ in $\ob{P}$ and $\ob{Q}$ respectively. There is a vast amount of literature analyzing the structure and topological significance of these graphs (see \cite{GLith} to begin).

In $\boundary N$ choose an annular neighborhood $A$ of $b$ and observe that $\boundary Q \cap A$ is the union of spanning arcs, one for each component of $\boundary Q \cap b$. Let $\mu = |\boundary Q \cap b|$. If $\ob{Q}$ is a closed surface, then $\mu = \Delta|\boundary Q|$. Following the orientation of $b$ label the components of $\boundary Q \cap A$ by $\lambda_1, \hdots, \lambda_\mu$. Around each vertex $v$ of $\Gamma_P$ we see the labels $\lambda_1, \hdots, \lambda_\mu$ or $\lambda_\mu, \hdots, \lambda_1$, according to the sign of intersection of the corresponding point of $\beta \cap \ob{P}$.  We say that two vertices of $\Gamma_P$ are \defn{parallel} if they have the same sign, and \defn{anti-parallel} if they have opposite signs. Thus, for two parallel vertices the labels run in the same direction and for antiparallel vertices, the labels run in opposite directions. 

A \defn{$\lambda_i$-cycle} in $\Gamma_P$ is a cycle $\sigma$ that can be given an orientation where the tail end of each edge is labelled $\lambda_i$. The number $l$ is the \defn{length} of the cycle. A \defn{great} $\lambda_i$-cycle is a cycle where each vertex in the cycle is parallel and a \defn{Scharlemann cycle} is a great $\lambda_i$-cycle bounding a disc $E_\sigma$ in $\ob{P}$ with interior disjoint from $\Gamma_P$. In this section, we show how a Scharlemann cycle in $\ob{P}$ gives a way of constructing a rational homology cobordism.
\begin{lemma}\label{Sch cycle implies rational cobordism}
Assume that $\alpha$ is a knot. Suppose that $\sigma$ is a Scharlemann cycle in $\Gamma_P$. Then there is a rational homology cobordism from $\ob{Q}$ to a surface $\ob{R}$ constructed as in Lemma \ref{Alg. Top.}. The surface $\ob{R}$ intersects $\alpha$ in two fewer points than does $\ob{Q}$.
\end{lemma}
\begin{proof}
Since the labels around each vertex of $\Gamma_P$ are oriented in the same direction, each edge of the Scharlemann cycle has labels $\lambda_{j}$ and $\lambda_{j+1}$ at its endpoints (for some $j$, with the indices running modulo $\mu$). Let $B$ be the closure of a component of $A - \boundary Q$ lying between the $j$th and $(j+1)$st components of $\boundary Q$. Name those components $\boundary_j Q$ and $\boundary_{j+1} Q$ respectively. (Note that if $K$ is an arc, then in $\boundary Q$, the arcs $\boundary_j Q$ and $\boundary_{j+1} Q$ may belong to the same component of $\boundary Q$.) Since there is an edge of $D \cap Q$ joining $\boundary_j Q$ to $\boundary_{j+1} Q$, those arcs pass through $A$ in opposite directions. (This uses the orientability of $P$ and $Q$.)

Thicken $B$ so that it is a 1-handle with attaching discs on $\ob{Q}$. Let $\mathbf{H}$ be the surface resulting from attaching the 1-handle to $\ob{Q}$ and pushing in the direction of the 1-handle to make it disjoint from $\ob{Q}$. The observation of the previous paragraph guarantees that the ends of the 1-handle are attached on the same side of $\ob{Q}$, and so $\mathbf{H}$ is orientable. The disc $D = E_\sigma$ is a compressing disc for $\mathbf{H}$ that lies outside the region between $\ob{Q}$ and $\mathbf{H}$. Let $\wihat{R}$ be the result of compressing $\mathbf{H}$ using $E_\sigma$ and then isotoping it off $\mathbf{H}$ away from $\ob{Q}$. By Lemma \ref{Alg. Top.}, the region $W$ between $\ob{Q}$ and $\ob{R}$ is a rational homology cobordism and $\ob{R}$ intersects $\alpha$ in two fewer points than does $\ob{Q}$.
\end{proof}
\begin{remark}
We need to assume that $\alpha$ is a knot for this lemma, for if $\alpha$ were an arc, there is no guarantee that the 2-handle produced by the Scharlemann cycle is disjoint from $\boundary M'$. The paper \cite{T1} examines that possibility in greater detail.
\end{remark}

\subsection{Using Gabai discs to find Scharlemann cycles}
Continue using the notation of the previous section, but now suppose that $\ob{P}$ is a Gabai disc for $Q$. For simplicity, let $D = \ob{P}$ and $G = \Gamma_P$. Our exposition is closely modelled on that in \cite[Section 2.5]{CGLS}. We allow $\alpha$ to be an arc, as before.

\begin{lemma}\label{Sch cycles exist}
There is a Scharlemann cycle in $G$.
\end{lemma}
\begin{proof}
Consider $G = \eta(K \cap D) \cup (\boundary Q \cap D)$ as a graph in $D$ (with fat vertices). Each edge of $G$ has a label at each endpoint that lies in the vertices of $G$. Some edges may have both endpoints on $\boundary D$ and some edges may have a single endpoint in the vertices of $G$ and a single endpoint on $\boundary D$. 

Each boundary component of $Q$ inherits an orientation from $Q$ and so each intersection point of $\boundary Q$ with a meridian of $K$ receives a sign.  Two intersection points have opposite signs if and only if the corresponding arcs of $\boundary Q \cap A$ traverse $A$ in opposite directions. An arc of intersection $D \cap Q$ has endpoints of opposite sign. Consequently

\textbf{Observation 1:} No edge of $G$ has endpoints with the same label. 

Since the number of boundary edges is strictly less than $\mu$, there exists a label $\lambda_i$ such that no boundary edge of $G$ has label $\lambda_i$ at its non-boundary endpoint. Since there are only a finite number of vertices, since they all have the same labelling scheme, and since no edge has the same label on both endpoints, we have:

\textbf{Observation 2:} There exists a $\lambda_i$ cycle in $G$.

Let $\sigma$ be an innermost such cycle. That is, $\sigma$ is a $\lambda_j$ cycle for some $j$ and $\sigma$ bounds a disc $E \subset D$ whose interior does not contain a $\lambda_k$ cycle for any $k$. Furthermore, choose $\sigma$ so that out of all such cycles, the interior of $E$ contains the fewest number of edges. 

If $E$ has the property that any edge of $G \cap E$ with an endpoint in a vertex of $\sigma$ lies in $\sigma$, then either the interior of $E$ is disjoint from $G$ (in which case, $\sigma$ is a Scharlemann cycle) or the disc obtained by removing a collar neighborhood of $\boundary E$ from $E$ is a Gabai disc for $Q$. If that were the case, then by Observation 2, we would have contradicted the minimality of $\sigma$.  Thus,

\textbf{Observation 3:} Either $\sigma$ is a Scharlemann cycle or some edge of $(G \cap E) - \sigma$ has an endpoint on a vertex of $\sigma$.

Assume that $\sigma$ is not a Scharlemann cycle. Since all the labels around vertices of $G$ run in the same direction, for each vertex $v \in \sigma$, either the edge $e(v,j-1)$ with label $\lambda_{j-1}$ or the edge $e(v,j+1)$ with label $\lambda_{j+1}$ lies in $E$ (indices run mod $(\mu+1)$).  Without loss of generality we may assume that it is $e(v,j+1)$. By Observation 3, there exists a vertex $v \in \sigma$, such $e(v,j+1)$ does not lie in $\sigma$. 

Consider the disc $E'$ obtained by slighly enlarging $E$ so that $\sigma$ is on the interior of $E'$. Since no edge of $G \cap E'$ with endpoint at a vertex of $G \cap E'$ has label $\lambda_{j+1}$ at the vertex, the reasoning that led us to Observation 2 shows that there exists a $\lambda_{j+1}$-cycle $\sigma'$ in $E'$. Since the edges of $G \cap E'$ completely contained in $E'$ are, in fact, contained in $E$. Since $e(v,j+1)$ does not lie in $\sigma$, the cycle $\sigma'$ is not equal to $\sigma$ and we have contradicted the minimality of $\sigma$.
\end{proof}

As a corollary, we obtain:

\begin{lemma}\label{Rationally Essential}
Suppose that $a$ and $b$ lie on a torus component of $\boundary N$ and that $N$ is hyperbolic and does not contain an essential surface of genus $g$. If $M'$ contains an essential surface of genus $g$, then either there is such a surface $\ob{Q}$ (in the same rational cobordism class) rationally cobordant to an inessential surface $\ob{R}$ intersecting $\alpha$ in two fewer points, or there is such a surface $\ob{Q}$ (in the same rational cobordism class) such that $Q = \ob{Q} \cap N$ is essential in $N$ and does not have a Gabai disc in $M$. Furthermore, in the former case $\ob{Q}$ is rationally inessential.
\end{lemma}
\begin{proof}
Let $\ob{Q} \subset M'$ be an essential surface of genus $g$ chosen so that out of all such surfaces in a given rational cobordism class, $\ob{Q}$ intersects $\alpha$ minimally. By hypothesis, this number is non-zero.

If $Q$ is inessential it is either compressible or boundary parallel. A compressing disc for $Q$ must have inessential boundary on $\ob{Q}$. Since $M'$ is irreducible, we may isotope $\ob{Q}$ through the 3-ball bounded by those two discs to a surface $\ob{R}$ intersecting $\alpha$ fewer times. This contradicts our choice of $\ob{Q}$.  If $Q$ is boundary parallel, it must be parallel to the torus component $T$ of $\boundary N$ containing $a$ and $b$. Since $Q$ is incompressible, it must be a boundary parallel annulus or disc. Since $\boundary Q \cap T$ is parallel to $a$, $Q$ cannot be a disc, so $Q$ is an annulus. Since $Q$ is a boundary-parallel annulus, $\ob{Q}$ is an inessential 2-sphere, contradicting our choice of $\ob{Q}$.

If there were a Gabai disc for $Q$ in $M$, then by Lemma \ref{Sch cycles exist}, there is a Scharlemann cycle in the graph $G$ in the Gabai disc. By Lemma \ref{Sch cycle implies rational cobordism}, $\ob{Q}$ is rationally cobordant to a surface $\ob{R}$ intersecting $\alpha$ in exactly two fewer points. The surface $\ob{R}$ must be inessential as otherwise we would have contradicted our choice of $\ob{Q}$. 
\end{proof}

\begin{corollary}\label{Sphere/Torus}
Suppose that $a$ and $b$ lie on a torus component of $\boundary N$ and that $N$ is hyperbolic. Then the following hold:
\begin{itemize} 
\item If $M'$ has an essential disc, then there is one $\ob{Q}$ such that $Q = \ob{Q} \cap N$ is essential in $N$ and does not have a Gabai disc in $M$.
\item If $M'$ has an essential sphere, then either there is one $\ob{Q}$ such that $Q = \ob{Q} \cap N$ is essential in $N$ and does not have a Gabai disc in $M$ or there is an essential sphere bounding a lens space summand of $M'$.
\item If $M'$ is irreducible and has an essential torus, then either there is one $\ob{Q}$ such that $Q = \ob{Q} \cap N$ is essential in $N$ and does not have a Gabai disc in $M$ or there is an essential torus in $M'$ bounding (possibly with a component of $\boundary M'$) a submanifold of Heegaard genus 2.
\item If $M'$ is irreducible, atoroidal, and has an essential annulus, then either there is one $\ob{Q}$ such that $Q = \ob{Q} \cap N$ is essential in $N$ and does not have a Gabai disc in $M$ or $\boundary M'$ is a single torus and $M'$ has Heegaard genus 2.
\end{itemize}
\end{corollary}
\begin{proof}
The proof is a continuation of the proof of Lemma \ref{Rationally Essential}.  If $Q$ is essential in $N$ and if there are no Gabai discs for $Q$ in $M$ we are done, so suppose that the surface $\ob{R}$ is rationally cobordant to $\ob{Q}$ and is inessential in $M'$. Since $\ob{Q}$ is incompressible, by Lemma \ref{Alg. Top}, no compressing disc for $\ob{R}$ can intersect the cobordism $W$ between $\ob{Q}$ and $\ob{R}$. Similarly, if $\ob{R}$ is boundary-parallel, $\ob{Q}$ lies outside the region of parallelism and if $\ob{R}$ is a sphere bounding a 3-ball, $\ob{Q}$ does not lie in the 3-ball.

If $\ob{Q}$ is a disc, then $\ob{R}$ is a boundary-parallel disc. Thus, $\boundary \ob{Q} = \boundary \ob{R}$ bounds a disc in $\boundary M'$. This contradicts the choice of $\ob{Q}$ to be a compressing disc for $\boundary M'$.

If $\ob{Q}$ is a sphere it bounds a lens space summand of $M'$ since $W$ is a punctured lens space. 

If $\ob{Q}$ is a torus, it lies outside the solid torus or the parallelism with $\boundary M'$ bounded by $\ob{R}$ and so it bounds (possibly with a component of $\boundary M'$) a submanifold of $M'$ of Heegaard genus 2.

If $\ob{Q}$ and $\ob{R}$ are annuli, their union is a torus $\ob{P}$ bounding a submanifold $W$ of $M'$ of Heegaard genus 2. Since $M'$ is atoroidal, $\ob{P}$ is inessential. Since $M'$ is irreducible, $\ob{P}$ is parallel to $\boundary M'$. Since $\boundary W = \ob{P}$, $\boundary M'$ is connected and $M'$ has Heegaard genus 2.
\end{proof}

\section{Applications}\label{Apps}
As in the introduction, let $N$ be a compact, connected, orientable 3-manifold and $T \subset \boundary N$ a torus component. Let $a$ and $b$ be slopes on $T$ with $\Delta \geq 1$. Let $(M',\alpha)$ and $(M,\beta)$ be the result of filling $T$ according to slopes $a$ and $b$ respectively, with $\alpha$ and $\beta$ the cores of the filling tori. We make the following assumptions:
\begin{enumerate}
\item[(A)] $M'$ does not have a lens space proper summand
\item[(B)] $M$ does not have a proper summand that is a rational homology ball of Heegaard genus at most 2 with torsion in first homology. 
\item[(C)] $H_2(M,\boundary M) \neq 0$.
\item[(D)] $N$ is irreducible and boundary-irreducible
\end{enumerate}

The case when both $M$ and $M'$ are reducible or boundary-reducible is thoroughly explored by Scharlemann \cite{S2}, so we concentrate on other types of exceptional fillings.

\begin{theorem}[The rationally essential theorem]\label{Main Thm 1}
Assume (A) - (D) and also that $M$ is reducible or that $(M,\beta)$ has an exceptional class. Suppose that $M'$ contains a rationally essential closed surface $\ob{Q}$ of genus $g$ and that no such surface rationally cobordant to $\ob{Q}$ is disjoint from $\alpha$. Then:
\[
(\Delta - 1)|\ob{Q} \cap \alpha| \leq -\chi(\ob{Q})
\]
\end{theorem}
\begin{remark}
The hypothesis that $N$ does not have more than 2 components that are of genus 2 or more is not critical as we can often reduce to that case by gluing a copy of $M'$ to $N$ along all the boundary components disjoint from both $\boundary \ob{Q}$ and a boundary compressing disc for $M$, if such exists. We do wonder, however, whether or not hypothesis (B) is needed.
\end{remark}
\begin{proof}
We may assume that $\ob{Q} \subset M'$ was chosen to minimize $|\ob{Q} \cap \alpha|$ out of all essential genus $g$ surfaces in $M'$ in the given rational cobordism class. By hypothesis, $\ob{Q} \cap \alpha \neq \nil$. By Lemmas \ref{Alg. Top.} and \ref{Rationally Essential}, either $Q$ is essential in $N$ and there is no Gabai disc for $Q$ in $M$ or $\ob{Q}$ is rationally inessential. 

By Theorem \ref{Lack-sutures}, we can choose sutures $\gamma \subset \boundary M$ so that $(N,\gamma)$ is a taut sutured manifold. In fact, we claim that we can do this so that the boundary component $T$ of $N$ corresponding to $\beta$ is disjoint from $\gamma$. To see that this is possible, double $N$ along $T$ and then apply Theorem \ref{Lack-sutures}. Since $T$ is an incompressible torus in the resulting manifold, cutting along it will not destroy tautness. If $M$ is a solid torus, we choose $\gamma$ differently: choose $\gamma$ to be any two parallel simple closed curves on $\boundary M$. Since $N$ is irreducible, $(M,\gamma,\beta)$ will be $\beta$-taut.

Consequently, $(M,\gamma,\beta)$ is a $\beta$-taut sutured manifold. Since $M$ is reducible or has an exceptional second homology class, the hypotheses of Theorem \ref{Sutured Mfld Thm} are satisfied. By that theorem, together with hypothesis (B) and the absence of Gabai discs for $Q$, we conclude that $I(Q) \geq 2\mu$. Since $\ob{Q}$ is a closed surface, $I(Q) = -2\chi(Q) + \Delta|\boundary Q|$.

Consequently,
\[
-2\chi(\ob{Q}) + 2|\boundary Q| \geq 2\Delta|\boundary Q|.
\]
Rewriting the inequality into the desired form is an easy exercise for your favorite schoolchild.
\end{proof}

\begin{theorem}[The exceptional surgery theorem]\label{Main Thm 2}
Assume the following:
\begin{itemize}
\item $N$ is irreducible and boundary-irreducible
\item $H_2(M,\boundary M) \neq 0$, $M$ is irreducible, and $(M,\beta)$ has an exceptional class
\end{itemize}
Then the following hold:
\begin{enumerate}
\item $M'$ is boundary-irreducible.
\item If $M'$ has an essential sphere, then $M'$ has a lens space proper summand.
\item If $M'$ is irreducible and has an essential torus, then either $\Delta = 1$ or there is an essential torus bounding (possibly with a component of $\boundary M'$) a submanifold of $M'$ of Heegaard genus 2.
\item Assume that $M'$ is irreducible and atoroidal and that $\boundary M'$ has at most two components of genus 2 or greater. Then if $M'$ has an essential annulus, either $\Delta = 1$ or $\boundary M'$ is a single torus and $M'$ has Heegaard genus 2.
\end{enumerate}

\end{theorem}

\begin{proof}
The proof is essentially the same as that of Theorem \ref{Main Thm 1}. By Lemma \ref{Sphere/Torus}, we may assume that there is an essential surface $\ob{Q}$ such that $Q = \ob{Q} \cap N$ has at least one boundary component on $\boundary \eta(\alpha)$, is essential in $N$, and has no Gabai discs in $M'$. The surface $\ob{Q}$ is a sphere, disc, annulus, or torus according to whether or not we are after conclusion (1), (2), or (3).

By Theorem \ref{Scharl-sutures}, we may pick sutures $\gamma$ disjoint from $\boundary \ob{Q}$.  If $M$ is a solid torus, as in Theorem \ref{Main Thm 1}, choose $\gamma$ to be any two parallel simple closed curves on $\boundary M$ disjoint from $\boundary \ob{Q}$.  As in the proof of Theorem \ref{Main Thm 1}, Theorem \ref{Sutured Mfld Thm} tells us that 
\[
-\chi(\ob{Q}) \geq |\boundary Q| (\Delta-1).
\]
If $\ob{Q}$ is a sphere or disc, then we contradict the fact that $\Delta \geq 1$. If $\ob{Q}$ is an annulus or torus, then we have $\Delta = 1$.
\end{proof}

\end{document}